\documentclass{amsart} 
\usepackage{amssymb,graphics,ifthen}
\usepackage[all]{xy} \xyoption{import}

\def \dim{{\mbox {dim}}\,}

\def\g{{\mathfrak g}}

\def\R{{\bf R}}

\def \OTO{{\rm O(2,1)}}
\def \OT3O{ \OTO \times \OTO \times \OTO }
\def \o21{o(2,1)}
\def \tOTO{{\rm \widetilde{O}(2,1)}}
\def \tOT3O{ \tOTO \times \tOTO \times \tOTO }

\def \o21{o(2,1)}

\def \PSLC{{\rm O_+(2,1)}}
\def \PS3LC{ \PSLC \times \PSLC \times \PSLC }
\def \SLC{{\rm \widetilde{O}_+(2,1)}}
\def \S3LC{ \SLC \times \SLC \times \SLC }

\def\dist {{\bf dist}}
\def\E{{\bf E}}
\def\d{{\rm d}}
\def\g{{g}}
\def\dvol{{\d \nu}}
\def\gauss{{\psi}}
\def\gb{{\tilde{g}}}
\def\a{{\mathcal A}}
\def\dv{\dvol}
\def\prior{{\eta}}
\def\briskk#1{\ifthenelse{\equal{#1}{}}{\tilde{R}_{\epsilon}}{\tilde{R}_{\epsilon}(#1)}}
\def\s{{s}}
\def\sr#1{{\mathfrak #1}}
\def\H{{\mathcal H}_{\epsilon}} 
\def\Ri{{\rm R}}
\def\scal{{\rm scal}}
\def\cov{{\mathfrak c}}

\newcommand{\normal}[2] { {\mathcal N}(#1,#2) }

\newcommand\Hom [1] {{\rm Hom(}#1{\rm )}}
\newcommand\Ric [1] {{\rm Ric}_{#1}}

\newtheorem{theorem}{Theorem}[section]
\newtheorem{thm}{Theorem}[section]

\newtheorem{lemma}[thm]{Lemma}
\newtheorem{proposition}[thm]{Proposition}

\newtheorem{corollary}[thm]{Corollary}

\begin{document}
\title[Bayesian estimation of maps]{A Bayesian approach to the
  Estimation of maps between riemannian manifolds}

\author[Butler, Levit]{Leo T. Butler and Boris Levit}
\address{LB: School of Mathematics, 6214 James Clerk Maxwell Building,
  The University of Edinburgh,  Edinburgh, UK, EH9 3JZ
BL: Department of Mathematics and Statistics
Queen's University,
Kingston, ON, Canada, K7L 3N6}
\email{l.butler@ed.ac.uk, blevit@mast.queensu.ca}

\date{\today}
\thanks{}

\maketitle

\begin{abstract}
Let $\Theta$ be a smooth compact oriented manifold without boundary,
imbedded in a euclidean space ${\bf E}^s,$ and let $\gamma$ be a
smooth map of $\Theta$ into a Riemannian manifold $\Lambda$. An
unknown state $\theta\in \Theta$ is observed via $X=\theta+\epsilon
\xi$ where $\epsilon>0$ is a small parameter and $\xi$ is a white
Gaussian noise. For a given smooth prior $\lambda$ on $\Theta$ and
smooth estimators $g(X)$ of the map $\gamma$ we derive a second-order
asymptotic expansion for the related Bayesian risk. The calculation
involves the geometry of the underlying spaces $\Theta$ and $\Lambda$,
in particular, the integration-by-parts formula. Using this result, a
second-order minimax estimator of $\gamma$ is found based on the
modern theory of harmonic maps and hypo-elliptic differential
operators.
\end{abstract}

\section{Introduction} 
\label{intro}

In many estimation problems, one has a state which lies on a manifold
but one observes this state plus some error in a euclidean space. It
is desirable to utilise the underlying geometry to construct an
estimator of the state. The present paper uses a Bayesian approach to
construct asymptotically minimax estimators along with the least favourable
Bayesian priors. 

The use of differential geometry in optimal statistical estimation has
a long history, as documented in a recent article ``Information
geometry'' on Wikipedia, for example. Early applications of
differential geometry to the derivation of second-order asymptotic
properties of the maximum likelihood estimates are summarized in
\cite{Amari}. However, a rigorous approach to second-order optimality
requires a decision-theoretical framework. This approach was developed
in \cite{Levit1,Levit2,Levit3} and a number of subsequent
publications.

In some cases, one is interested in the second-order optimal
estimation of a given function of parameters. For an early application
of this approach see \cite{Levit4}. As a general rule, such problems
require more sophisticated differential-geometric techniques such as
the theory of harmonic maps and hypoelliptic differential operators
\cite{EL:1971,Jost}.

Consider the following situation: $\E$ is a real $s$-dimensional
vector space with inner product $\sigma$ and $\Theta$
(resp. $\Lambda$) is a smooth manifold with riemannian metric ${\bf
g}$ (resp. ${\bf h}$). Assume that the smooth riemannian manifold
$(\Theta,{\bf g})$ is isometrically embedded in a euclidean space
$(\E,\sigma)$ via the inclusion map $\iota$, and $\Theta
\stackrel{\gamma}{\longrightarrow} \Lambda$ is a smooth map. Smooth
means infinitely differentiable. These data are summarized by the
diagram
$$
\xymatrix@!R@R=6pt{
(\E,\sigma) \ar@{.>}[ddrr]^{\g}\\
\\
(\Theta,{\bf g}) \ar@{->}[uu]^{\iota} \ar[rr]^{\gamma} && (\Lambda,{\bf h}),
}
$$ Suppose that $X \in \E$ is a gaussian random variable with
conditional mean $\theta \in \Theta$ and covariance operator%
\footnote{By convention, the covariance operator is the induced inner
  product on the dual vector space $\E^*$. If we regard
  $\sigma$ as a linear isomorphism of $\E \to \E^*$, then the
  covariance operator is the inverse linear isomorphism $\cov=\sigma^{-1} :
  \E^* \to \E$.}  $\epsilon^2 \cov$, {\it i.e.}
$$X \sim \normal{\theta}{\epsilon^2 \cov},\ \ \ \ \theta \in
\Theta.$$ A basic statistical problem is to determine an estimator of
``$\gamma(X)$,'' by which we mean an optimal extension of $\gamma$ off
$\Theta$, in the minimax sense. To make this precise, let
$\g : \E \to \Lambda$ be an estimator, and let $\dist$ be
the riemannian distance function of $(\Lambda,{\bf\ h})$. Define a
loss function by
\begin{equation} \label{eq:qr}
R_{\epsilon}(\g,\theta) = \int_{x \in \E}\,
\dist(\g(x),\gamma(\theta))^2\,
\gauss_{\epsilon}(x-\iota(\theta))\, \d x,
\end{equation}
where $|\bullet|$ is the norm on $\E$ induced by $\sigma$,
$\gauss_{\epsilon}(u) = \exp(-|u|^2/2\epsilon^2 ) / (2\pi
\epsilon^2)^{\frac{s}{2}}$ and $\d x$ is the volume form on $\E$
induced by $\sigma$.%
\footnote{One can introduce a $\sigma$-orthonormal coordinate system
  $x_i$ on $\E$. In this case, $|x|^2 = \sum_i x_i^2$ and $\d x = \d
  x_1 \wedge \cdots \wedge \d x_s$. }  Define the associated minimax
  risk
\begin{equation} \label{eq:mmrisk}
r_{\epsilon}(\Theta) = \inf_{\g}\, \sup_{\theta \in \Theta}\
R_{\epsilon}(\g,\theta).
\end{equation}

\subsection{Results} 
The present paper takes a Bayesian approach to the problem of
determining the asymptotically minimax estimator $\g$.  In
Bayesian statistics, the point $\theta$ is viewed as a random variable
with a prior distribution $\lambda(\theta) \d\theta$ where $\int_{\theta
\in \Theta} \lambda(\theta)\, \d\theta = 1$ ($\d\theta = \dvol_{{\bf
g}}$ is the riemannian volume of $(\Theta,{\bf g})$ ). The {\em
Bayesian risk} of a map $g$ is
\begin{equation} \label{eq:brisk}
R_{\epsilon}(g;\lambda) = \int_{x \in \E} \int_{\theta \in \Theta}
\dist(g(x),\gamma(\theta))^2\, \lambda(\theta)\,
\gauss_{\epsilon}(x-\iota(\theta))\ \d x\, \d\theta.
\end{equation}
A {\em Bayes estimator} $g : \E \to \Lambda$ is a map which minimizes
the Bayesian risk over all maps. In Theorem~\ref{thm:best}, an
expansion in $\epsilon$ of the Bayes estimator $\gb_{\epsilon}$, for a
fixed Bayesian prior, is computed. The constant term in
$\gb_{\epsilon}$ is $\gamma \circ \pi$, where $\pi : N \Theta \to
\Theta$ is the projection map of the normal bundle of $\Theta \subset
\E$. The order-$\epsilon^2$ term in $\gb_{\epsilon}$ is composed of
two parts: the first part is independent of the Bayesian prior and its
contribution is to reduce the energy of $\gamma$; the second part is
due to the gradient of the prior $\lambda$ and it tries to move the
estimator in the direction which maximizes
$\lambda$. Theorem~\ref{thm:best} also computes the Bayesian risk
$R_{\epsilon}(\gb_{\epsilon}; \lambda)$ of $\gb_{\epsilon}$ up to
$O(\epsilon^6)$.

The results of Theorem~\ref{thm:best} are used to obtain ``the''
optimal Bayesian prior. There arises a number of interesting problems
of a statistical nature in this regard: foremost is the problem of
deciding what should be the {\em flat} Bayesian prior. Given a flat
Bayesian prior, it is shown that the 2nd-order optimal Bayesian prior
satisfies an eigenvalue problem. This leads to a second difficulty: in
general, the leading term in the Bayesian risk is determined by $|\d
\gamma|^2$ and is therefore largely independent of the Bayesian
prior. Thus, Theorems \ref{thm:optprior}--\ref{thm:optpa} give several
criteria for second-order optimal Bayesian priors.

\subsection{Note to Reader} The present paper resulted from the work
of B. Levit \cite{Levit1, Levit2,Levit3,Levit4}. This work, and early drafts of
the present paper, were done largely in local coordinates using Taylor
series. This proved to be both daunting, difficult and unsatisfying
because we were forced to assume that $\Lambda$ was isometrically
embedded in some euclidean space and use the ambient distance
function. Paradoxically, these computations produced estimators which
did not take values in $\Lambda$.

The problem with the answer these computations produced was obvious,
the reason for the problem was less so. The ultimate reason is that
the Taylor series expansion of a function is not a tensorial, or
intrinsic, object. Rather, {\em a} Taylor series depends on the
geometry of the domain and range of the function: it is, in other
words, a {\em geometric} object. It is easy to see why this is: a
Taylor series requires the notion of a second derivative to be
defined, but it is well-known that a second derivative can be defined
only with the aid of an affine connection--a geometric
object. Calculations with Taylor series in local coordinates masked
this fact and completely mislead us.

This is a roundabout way of explaining the extensive geometric
formalism used in the present paper. We hope that the reader will
remember that behind this formalism is a simple aim: to define a
Taylor series in a rigorous and useful way. As a by-product, the
answers that result can be stated in a much more compact way.

This paper proceeds as follows: in section \ref{sec:mac}, a theory of
Taylor-Maclaurin series is developed for riemannian manifolds and
several useful curvature and integration-by-parts formulas are
developed that are used in seqsequent sections; section
\ref{sec:bayes} discusses the existence and uniqueness of a Bayes
estimator; section \ref{sec:expansion} utilizes the theory developed
in section \ref{sec:mac} to expand the Bayesian risk functional and
determined the Bayes estimator up to $O(\epsilon^6)$; section
\ref{sec:optp} develops criteria for second-order optimal Bayesian
priors in terms of the sub-laplacian of a naturally constructed
sub-riemannian structure; section \ref{sec:apps} computes the examples
where $\gamma$ is a riemannian immersion or submersion, which includes
the cases where $\gamma$ the identity map of $\Theta$ and the
inclusion map $\iota$ of $\Theta \subset \E$. 

Throughout, it is assumed that $\Theta, \Lambda$ are a compact,
connected, boundaryless smooth manifolds.

\section{Maclaurin Series} \label{sec:mac}
This section develops a theory of Maclaurin series of a map
between riemannian manifolds, then it exposes some useful formulas
from riemannian geometry that are used in subsequent sections. First,
it is useful recall some constructions.

\subsection{Induced metrics} \label{ssec:im}
Let $X$ and $Y$ be real inner-product spaces. The vector space of
linear maps $X\to Y$ is denoted by $\Hom{X;Y}$. Define the inner
product of linear maps $A,B \in \Hom{X;Y}$ by
\[
\langle A,B \rangle := \sum_i \langle A.e_i, B.e_i \rangle =
 {\rm Tr}(A'B),
\]
where $e_i$ is an orthonormal base of $X$. The Hilbert-Schmidt norm of
a linear map is defined in the natural way from this inner product. By
construction, if $x \in X$, then $|A.x| \leq |A||x|$.

We can make $X^{\otimes^n}$ (the $n$-fold tensor product of $X$ with
itself) into a real inner-product space by defining
\[
\langle a_1\otimes \cdots \otimes a_n, b_1 \otimes \cdots \otimes b_n
\rangle = \langle a_1,b_1 \rangle \cdots \langle a_n,b_n \rangle,
\]
for $a_i, b_i \in X$ and extending by bi-linearity. The previous
construction makes $\Hom{X^{\otimes^n}, Y}$ into a real inner-product
space. We will use these constructions henceforth without further
comment.

\subsection{Maclaurin series}
Let $(M,g)$ and $(N,h)$ be riemannian manifolds without boundary and
let $M \stackrel{\phi}{\longrightarrow} N$ be a smooth map. For $x \in
M$ and $y=\phi(x)$, let $T_xM$ (resp. $T_yN$) be the tangent space to
$M$ (resp. $N$) at $x$ (resp. $y$). The exponential map $\exp_x$
(resp. $\exp_y$) of $g$ (resp. $h$) is injective on a disk of radius
$a=a(x)$ in $T_xM$ (resp. $s=s(y)$ in $T_yN$), while the tangent map
of $\phi$ at $x$, $\d_x \phi$, maps a disk of radius $t$ into a disk
of radius $t \times |\d_x \phi|$. If $r=r(x)$ is defined to be the
minimum of $s(y)/|\d_x \phi|$ and $a(x)$, then there is a commutative
diagram%
\footnote{The functions $a$ and $s$ may be assumed to be smooth.}
$$
\xymatrix@!R@R=6pt{
(T_x^rM,g_x) \ar[rr]^{\varphi} \ar[dd]_{\exp_x} && (T_y^sN, h_y) \ar[dd]^{\exp_y}\\
&&& { i.e.}\ \varphi = \exp_y^{-1} \circ \phi \circ \exp_x,\\
(M,g) \ar[rr]^{\phi} && (N,h),
}
$$ where $T^r_xM$ (resp. $T^s_yN$) is the disk radius $r$ (resp. $s$)
in $T_xM$ (resp. $T_yN$) centred at $0$.  The map $\varphi$ is a
smooth map between open subsets of euclidean spaces that maps $0$ to
$0$. Its Maclaurin series expansion is well-defined and can be written
as
\begin{equation} \label{eq:varphi}
\varphi(v) = \d\varphi(v) + \frac{1}{2} \nabla\d\varphi(v,v) +
\frac{1}{6} \nabla^2\d\varphi(v,v,v) + O(|v|^4),
\end{equation}
for all $v \in T^r_xM$. The hessian $\nabla\d\varphi$ may be
understood as the ordinary second derivative of a map between vector
spaces, as can $\nabla^2\d\varphi$. However, Lemma \ref{lem:nabla} is
essential and relates the derivatives of $\d \varphi$ to the covariant
derivatives of $\d\phi$ \cite{EL:1971}.

\begin{lemma} \label{lem:nabla}
Let $v \in T_xM$. Then $\left. \nabla^k \d\varphi(v,\ldots,v)\right|_0
  = \left. \nabla^k \d\phi(v,\ldots,v)\right|_x$ for all $k\geq 0$ and
  all $x\in M$.
\end{lemma}

\begin{proof} Since $\exp_y \circ \varphi = \phi \circ \exp_x$ on the
  open set $T^r_xM$, it follows that $\left. \nabla^k \d(\exp_y \circ
  \varphi)\right|_0 = \left. \nabla^k \d(\phi \circ \exp_x)
  \right|_x$. It suffices to show that the lefthand side equals
  $\left. \nabla^k \d\varphi\right|_0$ and the righthand side equals $
  \left. \nabla^k \d\phi\right|_x$ when each are evaluated at
  $(v,\ldots,v)$. The chain rule, along with $\d_0\exp_y = {\rm
  id}_{T_yN}$, shows that
  $$\left. \nabla^k \d(\exp_y \circ \varphi)\right|_0 =
  \left. \nabla^k\d\varphi\right|_0 + T,$$
where $T$ is a sum of terms which are composition of forms
  $\nabla^l\d\varphi,\ \nabla^m \d\exp_y$ with $l,m<k$ and $m \geq
  1$. It therefore suffices to show that \\
\noindent{\bf Claim.} $\left. \nabla^m \d\exp_y\right|_0 = 0$ for all
  $m \geq 1$.\\ Let $v \in T_yN \equiv T_0(T_yN)$, let $c(t) =
  \exp_y(tv)$ be the unique geodesic passing through $v$, and let
  $m_v(t) = tv$ be the multiplication-by-$v$ map. Since $c(t) = \exp_x
  \circ m_v(t)$, $v = \d m_v(\partial_t)$ and $\nabla \d m_v = 0$, we
  have that $\nabla_{\dot{c}(t)} \dot{c}(t) = \nabla \d c( \partial_t,
  \partial_t ) = \nabla \d \exp_x (v,v)$. Thus $\nabla \d \exp_x(v,v)|_0
  = 0$.

In the general case, for $m\geq 2$, $\left. \nabla^m
\d\exp_y(v,\ldots,v)\right|_0 = \left. \nabla_{\dot{c}(t)} \left(
\cdots \left( \nabla_{\dot{c}(t)}\dot{c}(t) \right) \cdots \right)
\right|_{t=0}$. Since the innermost term vanishes identically in $t$,
the whole expression vanishes.

The proof for $\phi \circ \exp_x$ is similar. \end{proof}

\begin{lemma}
For all $v \in T^r_xM$, 
\begin{equation} \label{eq:ts}
\phi \circ \exp_x(v) = \exp_{\phi(x)}\left( \d
\phi(v) + \frac{1}{2}\nabla\d\phi(v,v) + \frac{1}{6}
\nabla^2\d\phi(v,v,v) + O(r^4)\right).
\end{equation}
\end{lemma}

\medskip
\noindent
{\bf Remarks}. (1) In general, the exponential map of $(N,h)$ is not a
global diffeomorphism. Consequently, $\varphi$ may not be globally
well-defined and its Maclaurin series (\ref{eq:ts}) need not converge
globally. The conjugate points of $\exp_{y}$ are obstructions to
global convergence. If $(N,h)$ is a simply-connected, non-positively
curved manifold, then $\exp_y$ is a global diffeomorphism and
$\varphi$ is globally defined. (2) If $v \in T_x M$ is a gaussian with
covariance operator $\epsilon^2 g_x$, then, since $\varphi$ is defined
on an open neighbourhood of $0$ and $\epsilon$ is small, its expected
value is essentially well-defined and equals, by Lemma \ref{lem:nabla}
and equation (\ref{eq:varphi}),
\[
\frac{1}{2}\, \epsilon^2\, \tau(\phi)_x + O(\epsilon^4)
\]
where $\tau(\phi)$ is the trace of the hessian
$\nabla\d\phi$. Riemannian geometers call $\tau(\phi)$ the {\em
tension field} of $\phi$. The tension field has an interesting
interpretation: if one views
\begin{equation}
\phi \mapsto \int_M |\d \phi |^2\, \dvol_g \ \ (\dvol_g = {\rm
  riemannian\ volume\ form\ of\ } g)
\end{equation}
as the energy of the map $\phi$, then $\phi \mapsto -\tau(\phi)$ is
the gradient of this functional. It is known that $\tau$ is a
semilinear elliptic differential operator that is analogous to the
laplacian \cite{EL:1971}. If one inspects the formula for the bayesian
estimator $\gb_{\epsilon}$ in Theorem \ref{thm:best}, one observes
that--neglecting the contribution of the prior $\lambda$--the
contribution of the $\frac{1}{2} \tau(\phi)$ is to move the estimate
$\gamma(\hat \theta)$ in the direction which reduces the energy
quickest. Indeed, if one includes both contributions, then their
combination can also be viewed in this fashion, but the energy
functional depends not on ${\bf g}$ and ${\bf h}$ but $u\cdot {\bf g}$
and ${\bf h}$ where the conformal factor $u$ is a fractional power of
$\lambda$.

\subsection{The Ricci Tensor} \label{ss:ric}
This section provides the key inputs to the proofs of Lemma
\ref{lem:intbp} and Proposition \ref{prop:nabladgo} by proving Lemmas
\ref{lem:ric}, \ref{lem:ric2} and the integration-by-parts formula in
Lemma \ref{lem:ibp2}. To do this, one must make an excursion into the
riemannian geometry of some naturally occurring vector bundles. In
this section, $(M,g)$ and $(N,h)$ are riemannian manifolds, possibly
with boundary, and $\phi : M\to N$ is a smooth map.

Let $\Hom{TM ; \phi^*TN}$ be the vector bundle of fibre-linear maps
between $TM$ and and $\phi^*TN$; a fibre $\Hom{TM ; \phi^*TN}_x, x \in
M$ of $\Hom{TM ; \phi^*TN}$ is the vector space of linear maps from
$T_xM$ to $T_{\phi(x)}N$. That is, $\Hom{TM ; \phi^*TN}_x = \Hom{T_xM
; T_{\phi(x)}N}$. One can view $\d \phi$ as a smooth section of
$\Hom{TM ; \phi^*TN}$. There is a natural metric connection on
$\Hom{TM ; \phi^*TN}$, which is denoted by $\nabla$ or
$\nabla^{\Hom{TM ; \phi^*TN}}$, that is induced by the (Levi-Civita)
connections on $TM$ and $TN$ respectively. Consequently, $\nabla \d
\phi$ is a smooth section of $T^* M \otimes \Hom{TM ; \phi^*TN}$. This
latter vector bundle admits a natural metric connection, in turn, and
$\nabla \nabla \d \phi = \nabla^2 \d\phi$ is then a smooth section of
$T^* M \otimes T^* M \otimes \Hom{TM ; \phi^*TN}$. In other words,
$\nabla^2 \d \phi$ is a 2-form with values in the vector bundle
$\Hom{TM ; \phi^*TN}$. This 2-form has a unique decomposition into a
symmetric and anti-symmetric part, {\it viz.}
$$\nabla^2 \d \phi (x,y) = \frac{1}{2} \left( \nabla^2_{x,y} +
\nabla^2_{y,x} \right) \d \phi + \frac{1}{2} \left( \nabla^2_{x,y} -
\nabla^2_{y,x} \right) \d \phi,$$ where $x,y$ are vector fields on $M$
and $\nabla^2_{x,y} \d\phi = \nabla_x( \nabla_y \d\phi) -
\nabla_{\nabla_xy} \d\phi$. Twice the anti-symmetric part of $\nabla^2
\d \phi$ is the curvature tensor of $(\Hom{TM ; \phi^*TN},\nabla)$ and is written as%
\footnote{A nice concise introduction to the subject of this
  paragraph is the monograph by Eells and
  Lemaire~\cite{EL:1971}. Their curvature tensor is minus that
  presented here, however. Their Ricci tensor is the same as that
  here.}
\begin{equation}
R_{x,y} \d \phi = \left( \nabla^2_{x,y} - \nabla^2_{y,x} \right) \d
\phi.
\end{equation}
There is a naturally-defined Ricci tensor associated to the curvature
$R$. Let $e_j$ be an orthonormal frame. Then for any tangent vector
$x$
\begin{equation}
\Ric{\d\phi}(x) = \sum_j R_{x,e_j} \d\phi \cdot e_j
\end{equation}
which is easily seen to be independent of the choice of orthonormal
frame. The Ricci tensor $\Ric{\d\phi}$ is a section of $\Hom{TM ; \phi^*TN}$, like
$\d\phi$.

The metric on $\Hom{TM,\phi^*TN}$ and associated bundles is described
in section \ref{ssec:im}.

\begin{lemma} \label{lem:expvalue}
Let $v \in T_pM$ be a gaussian with covariance operator (=metric)
$g_p$ and expected value $0$. The expected value of
\begin{enumerate}
\item $v \mapsto |\d\phi(v)|^2$ equals $|\d\phi|^2$;
\item $v \mapsto |\nabla \d\phi(v,v)|^2$ equals $|\tau(\phi)|^2 +
  2|\nabla\d\phi|^2$;
\item $v \mapsto \langle \d\phi(v), \nabla^2
  \d\phi(v,v,v) \rangle$ equals
  $$\sum_{i,j} \langle \d\phi(e_i),
  \nabla^2_{e_j,e_j} \d\phi \cdot e_i + \left( \nabla^2_{e_i,e_j} +
  \nabla^2_{e_j,e_i} \right) \d\phi \cdot e_j \rangle,$$
where $e_i$ is any orthonormal basis of $T_pM$.
\end{enumerate}
\end{lemma}

\noindent
It is recalled that the tension field $\tau(\phi)$ equals $\sum_i
\nabla\d\phi(e_i,e_i)$ and is a section of $\phi^* TN$ with its
induced norm. The norm of the second fundamental form $\nabla \d\phi$
is the norm of a section of $T^* M \otimes \Hom{TM ; \phi^*TN}$, so $|\nabla\d\phi|^2
= \sum_{i,j} |\nabla\d\phi(e_i,e_j)|^2$.

\begin{proof} A simple calculation. \end{proof}

\begin{lemma} \label{lem:ric}
$$\sum_{i,j} \langle \d\phi(e_i),
  \nabla^2_{e_j,e_j} \d\phi \cdot e_i + \left( \nabla^2_{e_i,e_j} +
  \nabla^2_{e_j,e_i} \right) \d\phi \cdot e_j \rangle_p = \langle
  \d\phi, 3\nabla \tau(\phi) - 2\Ric{\d\phi} \rangle_p. \eqno (*)$$
\end{lemma}

\newcommand\Trace [1] {{\rm Trace(}#1{\rm )}}
\begin{proof} Let $e_i$ be an orthonormal frame at $p$ and let $\dagger$ denote
  the left-hand side of (*). The subscript $p$ is dropped in the
  following. A computation yields
$$\nabla^2_{e_j,e_i} \d\phi \cdot e_j = \nabla^2_{e_j,e_j} \d\phi
  \cdot e_i \qquad \forall i,j. $$ If $\sum \langle
  \d\phi(e_i),\nabla^2_{e_j,e_i}\d\phi \cdot e_j \rangle$ is added and
  subtracted to $\dagger$, then one obtains
$$\dagger = \sum_{i,j} 3 \langle \d\phi(e_i),
  \nabla^2_{e_j,e_j} \d\phi \cdot e_i \rangle + \langle \d\phi(e_i),
  R_{e_i,e_j} \d\phi \cdot e_j \rangle,$$
which simplifies to
$$\dagger = \langle \d\phi, 3\Trace{\nabla^{2}\d\phi} + \Ric{\d\phi}
\rangle,$$ where $\Trace{\nabla^{2}\d\phi} = \sum_j
\nabla^{2}_{e_j,e_j}\d\phi$.  The identities $-\Delta \d\phi =
\Trace{\nabla^2\d\phi} + \Ric{\d\phi}$ and $-\Delta\d\phi = \nabla
\tau(\phi)$ yield the lemma~\cite{EL:1971}.
\end{proof}

\medskip
\noindent
The scalar $\langle \d\phi, \Ric{\d\phi} \rangle$ is simplified in the
following lemma. Let $\Ric{}^{M}$ be the Ricci tensor of $(M,g)$,
viewed as a section of $\Hom{TM,TM}$ and let ${\rm R}^N$ be the
Riemann curvature tensor of $(N,h)$. Let $e_i$ be an orthonormal frame
on $T_pM$, $u_i = \d\phi(e_i)$. A calculation shows
that~\cite{EL:1971}
\begin{equation} \label{eq:ricdphi}
\langle \d\phi, \Ric{\d\phi} \rangle = -\langle \d\phi, \d\phi(\Ric{}^{M}) \rangle + \sum_{i,j} \langle u_i, {\rm R}^N_{u_i,u_j} u_j \rangle.
\end{equation}
Since the second term is tensorial in $u_i$, this proves that

\begin{lemma} \label{lem:ric2}
If $\d\phi| T_pM = \d\hat{\phi} | T_pM$, then $\langle \d\phi,
\Ric{\d\phi} \rangle = \langle \d\hat{\phi}, \Ric{\d\hat{\phi}}
\rangle$ at $p$.
\end{lemma}

\begin{lemma} \label{lem:pi}
$\pi$ is harmonic: $\tau(\pi) = 0$.
\end{lemma}
\begin{proof}
Since $\pi \circ \iota = {\rm id}_{\Theta}$, and the second
fundamental form of $\iota$ is a quadratic form with values in
$N\Theta$, it follows that $\nabla \d\pi(\d\iota,\d\iota) = 0$,
i.e. $\nabla\d\pi\,|\, T\Theta$ vanishes. On the other hand, if
$\theta \in \Theta$ and $\eta \in N_{\theta}\Theta$, then
$\pi(\theta+s\eta) = \pi(\theta)$ for all $s$. Therefore
$\nabla\d\pi\,|\,N\Theta$ vanishes. These two facts show that the
trace of $\nabla\d\pi$, i.e. $\tau(\pi)$, vanishes.
\end{proof}

\subsection{Integration by Parts} \label{ss:ibp}

\noindent
This section recalls the integration-by-parts formula following the
discussion in \cite{EL:1971}. Let $\xi : V \to M$ be a vector bundle
over the riemannian $m$-manifold $(M,g)$ and let $\a^p$ be the space
of smooth sections of $\Lambda^p M \otimes V$, i.e. $\a^p$ is the
space of smooth $p$-forms on $M$ with values in $V$. Assume that $V$
is equipped with a metric and a compatible connection. There is a
natural metric connection on $\Lambda^p M \otimes V$, call it
$\nabla$, which induces an exterior derivation $\d : \a^p \to
\a^{p+1}$ by skew-symmetrization. Let $\d^* : \a^p \to \a^{p-1}$ be
the adjoint of $\d$ defined by
$$\int_M \langle \d \sigma, \rho \rangle\, \dvol_{g} = \int_M \langle
\sigma, \d^* \rho \rangle\, \dvol_{g} + \int_{\partial M} \left(\sigma
\wedge *\rho \right)\, \dvol_{g|\partial M}$$ for all $\sigma \in
\a^{m-p-1}$, $\rho \in \a^{p}$. Here $* : \a^p \to \a^{m-p}$ is the
Hodge star operator.

\begin{lemma} \label{lem:ibp}
Let $\lambda : M \to \R$ be a smooth function. Then
$$\int_M \lambda \langle \d \sigma, \rho \rangle\, \dvol_{g} = \int_M
\langle \sigma, \d^*(\lambda \rho) \rangle\, \dvol_{g} +
\int_{\partial M} \lambda \left(\sigma \wedge *\rho \right)\,
\dvol_{g|\partial M}.$$ In particular, if $\lambda | \partial M = 0$,
then
$$\int_M \lambda \langle \d \sigma, \rho \rangle\, \dvol_{g} = \int_M
\langle \sigma, \d^*(\lambda \rho) \rangle\, \dvol_{g}.$$
\end{lemma}

\begin{proof} This follows from applying the definition of $\d^*$ with $\rho' = \lambda \rho$. \end{proof}


\begin{lemma} \label{lem:ibp2}
Let $\d\phi \in \a^1$ be a $1$-form with values in $V = \phi^* TN$ and
assume that $\lambda | \partial M = 0$. Then
$$\d^*(\lambda \d\phi) = -\lambda\, \tau(\phi) - \d \phi(\nabla \lambda).$$ 
\end{lemma}

\begin{proof} 
Let $e_i$ be an orthonormal frame on $M$. If $\rho \in \a^1$ and
$\rho|\partial M = 0$, then $\d^*\rho = -\sum_i \nabla_{e_i} \rho
\cdot e_i$~\cite{EL:1971}. Thus, $\d^* (\lambda\d\phi) = - \sum_i
\lambda \nabla_{e_i} \d\phi \cdot e_i - \sum_i \d\phi(e_i) \cdot
\nabla_{e_i} \lambda$. The first term equals $\lambda \, \tau(\phi)$,
while the second term equals $\d\phi(\nabla \lambda)$.
\end{proof}

\section{Bayesian Estimators} \label{sec:bayes}

In Bayesian statistics, the point $\theta$ is viewed as a random
variable with a prior density $\lambda(\theta) \d\theta$ where
$\int_{\theta \in \Theta} \lambda(\theta)\, \d\theta = 1$ ($\d\theta =
\dvol_{{\bf g}}$ is the riemannian volume of $(\Theta,{\bf g})$
). Recall that the {\em Bayesian risk} of a map $g$ is defined to be
\begin{equation*} 
R_{\epsilon}(g;\lambda) = \int_{x \in \E, \theta \in \Theta}
\dist(g(x),\gamma(\theta))^2\, \lambda(\theta)\,
\gauss_{\epsilon}(x-\iota(\theta))\ \d x\, \d\theta.
\end{equation*}
A {\em Bayes estimator} $g : \E \to \Lambda$ is a map which
minimizes the Bayesian risk over all maps.

\subsection*{Existence and Uniqueness of the Bayes Estimator}
Let us sketch a proof of the existence and uniqueness of the Bayes
estimator. Define
\begin{equation} \label{eq:br}
h_{\epsilon}(x;g,\lambda) = \int_{\theta \in \Theta}
\dist(g(x),\gamma(\theta))^2\, \lambda(\theta)\,
\gauss_{\epsilon}(x-\iota(\theta))\, \d \theta.
\end{equation}
One sees that $R_{\epsilon}(g;\lambda) = \int_{x\in\E}
h_{\epsilon}(x;g,\lambda)\, \d x$. It is clear that a Bayes estimator
$\gb_{\epsilon}$ with prior $\lambda$, if it exists, will have the
property that $h_{\epsilon}(x;g,\lambda) \geq
h_{\epsilon}(x;\gb_{\epsilon},\lambda)$ for all $x$ and all estimators
$g$. One may assume that the class of estimators is the set of $L^1$
maps between $(\E,\d x)$ and $\Lambda$.

By compactness of $\Theta$, there is an $r_o > 0$ and $\epsilon_o>0$,
such that for all $x\in\E, \theta\in \Theta,$ and
$\epsilon<\epsilon_o$, the measure $\gauss_{\epsilon}(x-\iota(\theta))
\d \theta$ is supported, up to a remainder of $O(\exp(-1/\epsilon)$,
on the ball of radius $r_o$ about $\hat \theta=\pi(x)$. This
observation is trivial if $x$ lies within a distance $r$ of $\Theta$;
and it is trivial if $x$ lies in the complement of this neighbourhood,
since then the measure itself is $O(\exp(-1/\epsilon^2))$.

Let $B_{r_o}(\hat \theta)$ be the closed ball of radius $r_o$ centred
at $\hat \theta$. Possibly after shrinking $r_o$, the continuity of
$\gamma$ and compactness of $\Theta$ imply that the image,
$\gamma(B_{r_o}(\hat \theta))$, of $B_{r_o}(\hat \theta)$ may be
assumed to lie in a closed ball $D_{s_o}(\gamma(\hat \theta))$ of
fixed radius $s_o$ about $\gamma(\hat \theta)$. 

Introduce normal coordinates at $\hat \theta$ and $\gamma(\hat
\theta)$ so the above described balls are isometric to a ball about
$0$ in a real vector space with an almost euclidean riemannian metric
of the form $\sum_i \d x_i \otimes \d x_i + O(|x|^2)$.

We have therefore shown that $h_{\epsilon}(x;g,\lambda)$ may be
computed, up to a uniform remainder term of $O(\exp(-1/\epsilon))$,
using a map between two vector spaces that are equipped with metrics
that are euclidean up to second order. The techniques used in
\cite{Levit3,Levit4} can be used in this situation to show that the
Bayes estimator can be expanded as a formal power series in
$\epsilon^2$ and that this estimator is smooth.

\medskip
\noindent{\bf Remark.} If $(\Lambda, {\bf h})$ is a euclidean vector
space, the Bayes estimator exists and has the explicit form
\begin{equation} \label{eq:beucl}
\gb_{\epsilon}(x) = \frac{\int_{\theta\in \Theta} \gamma(\theta)\,
  \lambda(\theta)\, \gauss_{\epsilon}(x-\iota(\theta))\, \d
  \theta}{\int_{\theta\in \Theta} \lambda(\theta)\,
  \gauss_{\epsilon}(x-\iota(\theta))\, \d \theta}.
\end{equation}
This estimator has some rather curious properties: if $\gamma = \iota$
is the inclusion map $\Theta \subset \E$, then $\gb_{\epsilon}$ is the
weighted average of $\iota(\theta)$. Because this weighted average
need not lie on $\Theta$, one finds that the Bayes estimator is
somewhat unsatisfactory. The ultimate reason for this is the poor
choice of risk functional.
\medskip

\section{An Expansion of the Bayesian Risk} \label{sec:expansion}

First, introduce a change of variables.

\begin{lemma}
The map $J_{\epsilon} : \Theta \times \E \to \Theta \times \E$
defined by $\theta=\theta, x=\iota(\theta) + \epsilon z$ is a
diffeomorphism such that
\begin{equation} \label{eq:brisk2}
R_{\epsilon}(g;\lambda) = \int_{z \in \E, \theta \in \Theta}
\dist(g(\iota(\theta) + \epsilon z),\gamma(\theta))^2\, \lambda(\theta)\,
\gauss_{1}(z)\ \d z\, \d\theta.
\end{equation}
\end{lemma}

\begin{proof} A straightforward calculation. \end{proof}

The expression $\iota(\theta) + \epsilon z$ equals
$\exp_{\iota(\theta)} (\epsilon z)$ where $\exp$ is the exponential
map of the euclidean space $\E$. The Maclaurin series (equation
\ref{eq:ts}) implies that 
\begin{equation}
g \circ \exp_{\iota(\theta)} (\epsilon z) =
\exp_{g(\iota(\theta))}\left( \epsilon \d g(z) + \frac{1}{2}
\epsilon^2 \nabla \d g(z,z) + \frac{1}{6} \epsilon^3 \nabla^2 \d
g(z,z,z) + O(\epsilon^4 |z|^4) \right).
\end{equation}

\noindent
Since $\Lambda$ is connected, for each $a,b \in \Lambda$ there is a
geodesic $c : [0,1] \to \Lambda$ such that $c(0)=a$, $c(1)=b$ and the
length of $c$ is the distance between $a$ and $b$. That is, $|w|_a =
\dist(a,b)$ where $w=\dot{c}(0)$. The tangent vector $w=w_{a,b}$ is
not unique in general, but $w$ is a measurable function that is smooth
off the the cut locus of $a$. 

For $a=g(\iota(\theta))$ and $b=\gamma(\theta)$, let $w=w(\theta)$ be
the vector $w_{a,b}$. The vector $w(\theta)$ is characterized by the
property that $\exp_{g(\iota(\theta))}( w(\theta) ) = \gamma(\theta)$
for all $\theta$ and $w(\theta)$ is a shortest vector amongst all such
vectors. The Bayesian estimator $g_{\epsilon} : \E \to \Lambda$ is
written as
\begin{equation} \label{eq:geps}
g_{\epsilon}(x) = \exp_{g_o(x)}\left( \epsilon^2 g_2(x) +
O(\epsilon^4)\right).
\end{equation}
By definition, $g_{\epsilon}$ minimizes the Bayesian risk functional
$g \mapsto R_{\epsilon}(g;\lambda)$ for each $\epsilon$. Since the
Bayesian risk functional is an even function of $\epsilon$, the
Bayesian estimator is, too.

\begin{lemma} \label{lem:w}
Let the Bayesian risk $R_{\epsilon} = A_0 + O(\epsilon^2)$. Then
\begin{equation}
A_0(g;\lambda) = \int_{\theta \in \Theta} |w(\theta)|^2\
\lambda(\theta)\, \d\theta.
\end{equation}
Consequently, the Bayes estimator $g_{\epsilon}$ satisfies
\begin{equation} \label{eq:gepstheta}
g_{\epsilon} \circ \iota(\theta) = \exp_{\gamma(\theta)}(
\epsilon^2 g_2(\iota(\theta))) + O(\epsilon^4)) \qquad \forall \theta.
\end{equation}
 \end{lemma}

\begin{proof} The formula for $A_0$ is straightforward. Since $\lambda
  > 0$ a.e. by hypothesis, and $A_0 \geq 0$, it follows that $A_0 = 0$
    only if $w=0$ a.e., that is, only if $g_{\epsilon=0} \circ \iota =
    \gamma$. Since $g_{\epsilon=0} = g_o$, equation \ref{eq:geps}
    implies equation \ref{eq:gepstheta}.
\end{proof}

Lemma \ref{lem:w} shows that 
\begin{corollary} \label{cor:go}
$g_o \circ \iota = \gamma$. 
\end{corollary}

\begin{lemma} \label{lem:dgeps}
$\d g_{\epsilon} = \d g_o + \epsilon^2 \d g_2 + O(\epsilon^4).$
\end{lemma}

\begin{proof} Let $x \in \E$ and $v \in T_x \E$. It suffices to
  prove
$$\d_x g_{\epsilon} \cdot v = ( \d_x g_o + \epsilon^2 \d_x g_2 +
  O(\epsilon^4)) \cdot v. \eqno(*)$$
The left-hand side of (*) is
$$\d_x g_{\epsilon} \cdot v = \left. \frac{\d\ }{\d t} \right|_{t=0}\
\exp_{a(t)}\left( \epsilon^2 b(t) + O(\epsilon^4) \right), \eqno
(**)$$ where $a(t) = g_o(x+tv)$ is a curve in $\Lambda$ and $b(t) =
g_2(x+tv)$ is a curve of tangent vectors along $a(t)$. The right-hand
side of (**) is the Jacobi field $J(s)$ on $(\Lambda, {\bf h})$ with
initial conditions $J(0)=\dot{a}(0)$ and
$\dot{J}(0)=\dot{b}(0)+O(\epsilon^2)$ at the time
$s=\epsilon^2$. Since $J(s) = J(0) + s\dot{J}(0) + O(s^2)$ and
$\dot{a}(0) = \d_x g_o \cdot v$, $\dot{b}(0) = \d_x g_2 \cdot v$ we
see that (**) implies (*). \end{proof}

\begin{lemma} \label{lem:bexpansion}
Let the Bayesian risk of the Bayesian estimator $g_{\epsilon}$ be
$R_{\epsilon} =\epsilon^2 A_2 + \epsilon^4 A_4 + O(\epsilon^6)$. If
$A_k = \int a_k \lambda(\theta) \d \theta$, then the integrand $a_k$ is
\begin{eqnarray}
a_2 &=& |\d g_o|^2\\
a_4 &=& \left\{ \begin{array}{l} \frac{1}{4}
|\tau(g_o)|^2 + \frac{1}{2}| \nabla \d g_o |^2 + \langle \d g_o,
\nabla \tau(g_o) - \frac{2}{3} \Ric{\d g_o} \rangle +\\ |g_2|^2 + 2
\langle \d g_2, \d g_o \rangle + \langle g_2, \tau(g_o) \rangle
\end{array} \right. \label{eq:a4}
\end{eqnarray}
\end{lemma}

\begin{proof} 
When one expands $g_{\epsilon}(\iota(\theta)+\epsilon z)$ as a
Maclaurin series, one obtains
\begin{equation}
\exp_{\gamma(\theta)}\left(\epsilon^2 g_2+ \epsilon \d g_{\epsilon}(z)
+ \frac{1}{2} \epsilon^2 \nabla \d g_{\epsilon}(z,z) + \frac{1}{6}
\epsilon^3\ \nabla^2 \d g_{\epsilon}(z,z,z)
+ O(\epsilon^4 |z|^4) \right).
\end{equation}
Since $\d g_{\epsilon} = \d g_o + \epsilon^2 \d g_2 + O(\epsilon^4)$
by Lemma \ref{lem:dgeps}, the Maclaurin series equals
\begin{equation}
\exp_{\gamma(\theta)}\left( \epsilon \d g_o(z) + 
\epsilon^2 \left( g_2 + \frac{1}{2}\nabla \d
g_o(z,z) \right) + \epsilon^3 \left( \d
g_2(z) + \frac{1}{6}
\nabla^2 \d g_o(z,z,z) \right) + O(\epsilon^4 |z|^4) \right).
\end{equation}
The distance between $g_{\epsilon}(\iota(\theta)+\epsilon z)$ and
$\gamma(\theta)$ expands to
\begin{eqnarray}
\dist &=& \epsilon^2|\d g_o(z)|^2 + \epsilon^4\left( |g_2 + \frac{1}{2}
\nabla \d g_o(z,z)|^2 + 2\langle \d g_o(z), \d g_2(z) + \frac{1}{6}\nabla^2
\d g_o(z,z,z) \rangle \right) \nonumber\\ 
&&+ \epsilon^3 (\cdot) + \epsilon^5 (\cdot) +
O(\epsilon^6|z|^6),
\end{eqnarray}
where the coefficients on the odd powers of $\epsilon$ are odd
polynomials in $z$. Lemmas~\ref{lem:ric}--\ref{lem:ric2} now implies
this lemma.
\end{proof}

\noindent
Recall that $\pi : N\Theta \to \Theta$ is the normal bundle of
$\Theta$ in $\E$; the tangent bundle of $N\Theta$ is isometric to
$T_{\Theta}\E$ while $\d\pi$ is the orthogonal projection of
$T(N\Theta)$ onto $T\Theta$. Corollary \ref{cor:go} implies that $\d
g_o | T\Theta = \d\gamma$, so on $\Theta$ $|\d g_o|^2 \geq
|\d\gamma|^2$ with equality iff $\d g_o | N\Theta = 0$ or $\d g_o |
T_{\Theta} \E = \d \gamma \circ \d\pi$. By Lemma
\ref{lem:bexpansion}, these considerations show that

\begin{proposition} \label{prop:go}
The Bayesian estimator satisfies 
\begin{enumerate}
\item $g_o \circ \iota = \gamma$;
\item $\d g_o = \d(\gamma \circ \pi)$ on $T_{\Theta}\E$.
\end{enumerate}
\end{proposition}

\noindent
Define $\Gamma = \gamma \circ \pi$. The next step is to show that
$\nabla \d g_o = \nabla \d \Gamma$ on $\Theta$. To do so requires that
$a_4$ (Lemma \ref{lem:bexpansion}) be simplified.

\begin{lemma} \label{lem:intbp}
Under the standing hypothesis that $\lambda > 0$ on $\Theta$, we have

\begin{enumerate}
\item $\int_{\Theta} \d\theta\, \lambda(\theta)\, \langle \d g_o, \d
  g_2 \rangle = \int_{\Theta} \d\theta\, \lambda(\theta)\, \langle
  g_2, \tau(\gamma) + \d \gamma(\nabla \log \lambda) \rangle$;
\item $g_2 \circ \iota = \tau(\gamma) - \frac{1}{2}\tau(g_o) + \d
  \gamma( \nabla \log \lambda).$
\item $a_4 = \frac{1}{2}| \nabla
  \d g_o |^2 -  |\tau(\gamma)+\d\gamma(\nabla \log \lambda)|^2 - \frac{2}{3} \langle \d\Gamma, \Ric{\d\Gamma} \rangle;$
\end{enumerate}
\end{lemma}

\begin{proof}
The inner product $\lambda\, \langle \d g_o, \d g_2 \rangle$ on
$\Theta$ equals $\langle \lambda \cdot \d(\gamma \circ \pi), (\d g_2)
\circ \iota \rangle$ which equals $\langle \lambda \cdot \d\gamma, \d
(g_2 \iota) \rangle$. The integration-by-parts formula (Lemma
\ref{lem:ibp2}) for sections of $T^*\Theta \otimes \gamma^*T\Lambda$
yields (1).

(1) along with equation \ref{eq:a4} yields
$$a_4 = \left\{ \begin{array}{l} |g_2 \circ \iota - \tau(\gamma)
  - \frac{1}{2}\tau(g_o) + \d \gamma( \nabla \log \lambda)|^2 +\\ \langle
  \tau(g_o), \tau(\gamma) + \d\gamma(\nabla \log \lambda) \rangle -
  |\tau(\gamma)+\d\gamma(\nabla \log \lambda)|^2 +\\ \frac{1}{2}| \nabla
  \d g_o |^2 + \langle \d g_o, \nabla \tau(g_o) - \frac{2}{3} \Ric{\d
    g_o} \rangle;
\end{array} \right.$$
It is clear that $a_4$ is minimized by setting $g_2$ to that in (2). 

Finally, Lemma \ref{lem:ric2} implies that $\langle \d g_o, \Ric{\d
g_o} \rangle = \langle \d\Gamma, \Ric{\d\Gamma} \rangle$ on
$\Theta$. A second application of the integration-by-parts formula to
$\lambda \, \langle \d g_o, \nabla \tau(g_o) \rangle$ proves (3).
\end{proof}

\begin{proposition} \label{prop:nabladgo}
The Bayesian estimator satisfies
\begin{enumerate}
\item $\nabla \d g_o = \nabla \d\Gamma$ on $\Theta$;
\item $\tau(g_o) = \tau(\gamma)$ on $\Theta$;
\item $g_2 \circ \iota = \frac{1}{2}\tau(\gamma) + \d\gamma(\nabla \log \lambda)$; and
\item $a_4 = \frac{1}{2}| \nabla \d\Gamma |^2 -
  |\tau(\gamma)+\d\gamma(\nabla \log \lambda)|^2 - \frac{2}{3} \langle
  \d\Gamma, \Ric{\d\Gamma} \rangle.$
\end{enumerate}
\end{proposition}

\begin{proof} By (3) of Lemma \ref{lem:intbp}, it is clear that $a_4$
  is minimized iff $|\nabla \d g_o |^2$ is minimized. Let $\alpha$
(resp. $\beta$) be the orthogonal projection of $T_{\Theta} \E$ onto
$T\Theta$ (resp. $N\Theta$). This orthogonal decomposition yields the
equality
$$|\nabla \d g_o |^2 = |\nabla \d g_o(\alpha,\alpha) |^2 + 2 |\nabla
\d g_o(\alpha,\beta) |^2 + |\nabla \d g_o(\beta,\beta) |^2.$$ Since
$|\nabla \d g_o(\alpha,\alpha) |^2 = |\nabla \d
g_o(\d\iota,\d\iota)|^2$ and $\nabla \d(g_o\iota) = \nabla \d
g_o(\d\iota,\d\iota) + \d g_o \cdot \nabla \d\iota$, Proposition
\ref{prop:go} yields $|\nabla \d g_o(\alpha,\alpha) |^2 =
|\nabla\d\gamma|^2 = |\nabla\d(\gamma\pi)(\alpha,\alpha)|^2$. Part (4)
follows from this.

A Maclaurin series argument shows that Proposition \ref{prop:go}
implies that $\nabla \d g_o(\alpha,\beta) = \nabla \d(\gamma
\pi)(\alpha,\beta)$, while $|\nabla\d g_o(\beta,\beta)|$ is
unconstrained. This is minimized by $0 = |\nabla \d(\gamma
\pi)(\beta,\beta)|$. This proves (1).

The formula $\tau(g_o) = \tau(\gamma) + \d\gamma \cdot \tau(\pi)$ is
implied by (1). Since $\pi$ is harmonic by Lemma~\ref{lem:pi}, this
implies (2). Lemma \ref{lem:intbp} part (2) implies (3).
\end{proof}

\bigskip\noindent
Let us summarize the results of this section.

\begin{theorem} \label{thm:best}
Let $\gb_{\epsilon}(x) = \exp_{g_o(x)}\left( \epsilon^2 g_2(x) +
O(\epsilon^4)\right)$ be the Bayesian estimator for the Bayesian risk
functional $R_{\epsilon}$ (Equation \ref{eq:brisk}) with a fixed
Bayesian prior $\lambda>0$. Then
\begin{enumerate}
\item for all $x \in N \Theta$, where $\hat{\theta} = \pi(x)$ and
$|x-\hat{\theta}| \leq r$, $$\gb_{\epsilon}(x) =
\exp_{\gamma(\hat{\theta})}\left( \epsilon^2 \left( \frac{1}{2}
\tau(\gamma) + \d\gamma(\nabla \log\lambda) \right)_{\hat{\theta}} +
O(r\epsilon^4) \right).$$\\

\item
\begin{align*}
R_{\epsilon}(\gb_{\epsilon} ; \lambda) &= \epsilon^2\, \int \d\theta\,
\lambda\, |\d\gamma|^2 +\\ &\ \ \ \ \epsilon^4 \int \d\theta\,
\lambda\, \left\{ \frac{1}{2}| \nabla \d\Gamma |^2 -
|\tau(\gamma)+\d\gamma(\nabla \log \lambda)|^2 - \frac{2}{3} \langle
\d\Gamma, \Ric{\d\Gamma} \rangle \right\} + O(\epsilon^6),
\end{align*}
where $\Gamma =\gamma\pi$.
\end{enumerate}
\end{theorem}

\begin{proof} 
(1) Let $x \in \E$ and $|x-\Theta| \leq r$. By the hypothesis on the
  radius $r$, the orthogonal projection of $x$ onto $\Theta$ is
  well-defined; this orthogonal projection is denoted by $\hat
  \theta=\pi(x)$. Write $x=\iota(\hat \theta) + \epsilon z$, where by
  construction, $z \in N_{\hat \theta}\Theta$. The Maclaurin expansion
  of $\gb_{\epsilon}$ at $\iota(\hat \theta)$ gives
\begin{equation} \label{eq:gb}
\gb_{\epsilon}(x) = \gb_{\epsilon} \circ \exp_{\iota(\hat \theta)}(\epsilon
z) = \exp_{\gb_{\epsilon} \circ \iota(\hat \theta)}\left(\, \epsilon
\d \gb_{\epsilon}\cdot z + \frac{\epsilon^2}{2}\
\nabla\d\gb_{\epsilon}(z,z) + O(r^3 \epsilon^3)  \right)_{\hat \theta}.
\end{equation}
Equation (\ref{eq:gepstheta}) and Proposition \ref{prop:nabladgo}.3
show that $\gb_{\epsilon}\circ \iota(\hat \theta) = \exp_{\gamma(\hat
\theta)} w_{\epsilon}$ where $w_{\epsilon} = \epsilon^2 \times \left(
\frac{1}{2}\tau(\gamma) + \d \gamma(\nabla \log \lambda) \right)_{\hat
\theta} + O(\epsilon^4)$. On the other hand, since $z \in N_{\hat
\theta} \Theta$, Proposition \ref{prop:go} shows that
\begin{equation} \label{eq:dgeps}
\d \gb_{\epsilon} \cdot z = \epsilon^2 \times \d g_2\cdot z +O(r
\epsilon^4)
\end{equation}
while Proposition \ref{prop:nabladgo}.1 shows that
\begin{equation} \label{eq:ndgeps}
\left. \nabla \d \gb_{\epsilon}(z,z)  \right|_{\hat \theta} =
\epsilon^2 \times \nabla \d g_2(z,z) + O(r^2 \epsilon^4).
\end{equation}
Equations (\ref{eq:gb}--\ref{eq:ndgeps}) imply that $\gb_{\epsilon}(x)
  = \exp_{\gamma(\hat \theta)}(v_{\epsilon})$ and that $v_{\epsilon} =
  w_{\epsilon} + O(r \epsilon^3).$ Therefore
\begin{equation} \label{eq:gbeps}
\gb_{\epsilon}(x) = \exp_{\gamma(\hat \theta)} \left( \epsilon^2 \times \left( \frac{1}{2}\,
\tau(\gamma) + \d \gamma(\nabla \log \lambda )\right)_{\hat \theta} +
O(r \epsilon^3) \right).
\end{equation}
Since the Bayesian estimator $\gb_{\epsilon}$ is an even function of
$\epsilon$, the error is not $O(r \epsilon^3)$ but must be $O(r
\epsilon^4)$.

(2) This is a straightforward application of the preceding work.

\end{proof}

\medskip
\noindent{\bf Remark.} Inspection of the proof above shows that the
$O(r \epsilon^3)$ term in Equation (\ref{eq:gbeps}) is $\epsilon^3
\times \d g_2 \cdot z$. Thus, the proof also shows that $\d g_2 |
N_{\Theta} \Theta$ vanishes.
\medskip

\section{Optimal Priors} \label{sec:optp}
In this section we are interested in the behaviour of the mimimax risk
which can be defined as
$$r_\epsilon(\Theta)=\inf_{g_\epsilon}\sup_{\theta\in\Theta}\,
(R_\epsilon(g_{\epsilon},\theta)-\epsilon^2|d\gamma(\theta)|^2),$$ where the
$\inf$ is taken over all possible (sequences of) estimators
$g_\epsilon$. The problem of finding the asymptotic behaviour of
$r_\epsilon(\Theta)$ can be derived in a relatively straightforward
manner from the previous results. Essentially the problem reduces to
finding {\it optimal priors} maximizing the first non-trivial term of
the Bayes risk.

Since in the case of smooth functions the minimax risk
$r_{\epsilon}(\Theta)$ is typically of order $\epsilon^4$, we can
define the {\it second-order minimax risk} as
\begin{equation} \label{eq:reps}
r(\Theta)=\lim_{\epsilon\to
0}\inf_{g_\epsilon}\sup_{\theta\in\Theta}\,
\epsilon^{-4}(R_\epsilon(g_\epsilon,\theta)-\epsilon^2|d\gamma(\theta)|^2).
\end{equation}

Sometimes, a more general minimax risk may be of interest. Let $p$ and
$q>0$ be given function defined on $\Theta$. Then equation
(\ref{eq:reps}) can be modified as
$$r^*(\Theta)=\lim_{\epsilon\to
0} \inf_{g_\epsilon}\sup_{\theta\in\Theta}\,
\frac{R_\epsilon(g_\epsilon,\theta)-\epsilon^2|d\gamma(\theta)|^2-\epsilon^4p(\theta)}{\epsilon^{4}q(\theta)}.
$$ Even more useful is the following equivalent definition of the
second order minimax risk
\begin{equation} \label{eq:reps2}
r^*(\Theta)=\inf\{r|\exists g: R_\epsilon(g_\epsilon,\theta)\le
\epsilon^2|d\gamma(\theta)|^2+\epsilon^4(p(\theta)+r q(\theta))\}.
\end{equation}
The advantage of the last formula is that, unlike the previous one, it
allows consideration of smooth functions $q(\theta)\ge 0,\ q(\theta)\ne \rm
const.$ It is thus in this form that the second-order minimax risk
will be considered below.

Theorem \ref{thm:best}, part 2, gives a formula for the Bayesian risk
expanded up to $O(\epsilon^6)$ of the Bayesian estimator
$\gb_{\epsilon}$. One would like to determine the Bayesian prior
distribution that minimizes the Bayesian risk. There are a couple
interesting twists that arise at this point. First, the implicit {\em
flat} prior is a constant multiple of the riemannian volume form $\d
\theta$. However, there is no reason to single out the riemannian
volume form as the flat prior. Rather, one can introduce the flat
prior $\dv = a(\theta)\, \d \theta$ ($a>0$ a.e.) and the Bayesian
prior $\prior(\theta)\, \dv = \lambda(\theta)\, \d \theta$, where
$\prior = \lambda/a$. Let us stress that the change from $\d \theta$
to $\dv$, and $\lambda$ to $\prior$, does not change the foregoing
calculations and results. Second, the minimizers of the Bayesian risk
functional are also the minimizers of the functional
\begin{equation} \label{eq:brisk4}
\briskk{g;\lambda} = R_{\epsilon}(g;\lambda) - \epsilon^2
\int_{\Theta} \d \theta\, \lambda\, |\d \gamma|^2,
\end{equation}
since the second term is independent of $g$. One may, therefore, elect
to minimize the functional $\briskk{}$, to obtain the Bayesian
estimator $\gb_{\epsilon}$--which is implicitly a function of the
Bayesian prior $\lambda \d \theta$--and proceed to determine the
second-order optimal prior by minimizing $\lambda \mapsto
\briskk{\gb_{\epsilon}; \lambda}$. Finally, inspection of part (2) of
Theorem \ref{thm:best} shows that one needs tools to understand how to
simplify the term $|\tau(\gamma) + \d \gamma(\nabla \log
\lambda)|^2$. The requisite tool is known as {\em sub-riemannian}
geometry.

\subsection{Sub-riemannian geometry}
Let us describe a particular construction of a sub-riemannian
geometry. Let $(M,g) \stackrel{\phi}{\to} (N,h)$ be a smooth map, and
let $D_p = \ker \d_p\phi^{\perp}$ for $p \in M$. The collection $D =
\cup_p D_p$ is a singular distribution on $M$. It is equipped with an
inner product $\s$ -- a sub-riemannian metric -- by declaring that
$\d_p \phi | D_p \to {\rm im}\, \d_p \phi \subset T_pN$ is an
isometry. That is, $\s = \phi^* h | D$.

One may think of the subriemannian structure $(D,s)$ as a singular
distribution of directions in which one may travel, along with a
metric which allows one to measure speed (and angles). Subriemannian
structures arise in optimal control problems quite frequently
\cite{Mont,Bloch}.

One may equivalently characterize the sub-riemannian structure
$(D,\s)$ by a bundle map $\mu : T^*M \to TM$ such that (i) $\mu$ is
self-adjoint; and (ii) the image of $\mu$ equals $D$. In the present
context, the map $\mu$ is characterized by the identity
$$\mu(\d u, \d v) = \s(\nabla u, \nabla v) = \langle \d\phi(\nabla u),
\d\phi(\nabla v) \rangle,$$ for all smooth functions $u,v : M \to
\R$. Equivalently, $\mu \cdot \d u = \d\phi' \d \phi(\nabla u).$

An augmented sub-riemannian structure $\sr{D}=(D,\s,\dvol)$ is a
sub-riemannian structure $(D,\s)$ plus a volume form $\dvol$. The
augmented sub-riemannian structure permits one to define a {\em
sub-laplacian} $\Delta_{\sr{D}}$, which is a second-order,
self-adjoint differential operator.%
\footnote{Warning: the sign of $\Delta_{\sr{D}}$ conflicts with the sign in
Montgomery's exposition~\cite{Mont}, but it accords with the sign
convention in riemannian geometry.} In local coordinates
\begin{equation} \label{eq:sublaplocal}
\Delta_{\sr{D}} = -\sum_{ij} \frac{1}{f} \frac{\partial\ \ }{\partial x^i}
\left( f \cdot \mu^{ij} \frac{\partial\ \ }{\partial x^j} \right),
\end{equation}
where $\dvol = f \d x^1 \wedge \cdots \wedge \d x^m$.  The
sub-laplacian is defined invariantly by
\begin{equation} \label{eq:sublap}
\int u \cdot \Delta_{\sr{D}} v\, \dvol = \int \mu(\d u, \d v)\, \dvol,
\end{equation}
for all smooth functions that vanish on $\partial M$. The
self-adjointness of $\mu$ implies $\Delta_{\sr{D}}$ is self-adjoint.

If $a$ is a positive function, then let the augmented
sub-riemannian structure $(D,\s,a\cdot \dvol)$ be denoted by $a\cdot
\sr{D}$. Equation (\ref{eq:sublap}) shows that the sub-laplacian of
the augmented sub-riemannian structures differ by a differential
operator of first order
\begin{equation} \label{eq:sublapa}
\Delta_{a\cdot \sr{D}} = \Delta_{\sr{D}} - \mu \cdot \d \log a.
\end{equation}

\subsection{Optimal priors, I}
The discussion of sub-riemannian geometry allows the expansion of the
Bayesian risk (Theorem \ref{thm:best}). The term $\int \d\theta
\lambda |\tau(\gamma) + \d\gamma(\nabla \log \lambda)|^2$ expands to
\begin{equation} \label{eq:simp}
\int \d\theta \left\{ \omega^2 |\tau(\gamma)|^2 + 4 \omega \langle
\tau(\gamma), \d\gamma(\nabla \omega) \rangle + 4 \omega
\Delta_{\sr{E}} \omega \right\},
\end{equation}
 where $\lambda = \omega^2$ and $\sr{E}=(E,\s,\d \theta)$ where $E =
\ker \d\gamma^{\perp}$ and $\s = \gamma^* {\bf h}|E$. Define
\begin{align}
\kappa & = \frac{1}{2} |\nabla \d\Gamma|^2 - \frac{2}{3}\, \langle
\d\Gamma, \Ric{\d\Gamma} \rangle + |\tau(\gamma)|^2 + 2 \langle
\d\gamma, \nabla \tau(\gamma) \rangle, \label{eq:kappadef}\\
L & = 4 \Delta_{\sr{E}} + \kappa. \label{eq:Ldef}
\end{align}

\noindent
From this discussion, and an application of the integration-by-parts
formula to $\int \d\theta \lambda \langle \tau(\gamma),
\d\gamma(\nabla\log\lambda) \rangle$, the following is clear.

\begin{theorem} \label{thm:optprior}
The Bayesian risk functional at $\gb_{\epsilon}$ with prior
$\lambda=\omega^2$ equals
\begin{equation} \label{eq:optprior}
R_{\epsilon}(\gb_{\epsilon};\omega^2) = \epsilon^2 \int \d\theta\,
\omega^2\, |\d\gamma|^2 + \epsilon^4 \int \d\theta\, \omega\cdot
L\omega + O(\epsilon^6),
\end{equation}
while
\begin{equation} \label{eq:optprior2}
\briskk{\gb_{\epsilon};\omega^2} = \epsilon^4 \int \d\theta\, \omega\cdot
L\omega + O(\epsilon^6).
\end{equation}
\end{theorem}

\bigskip\noindent 
Define a differential operator $\H$ on $\Theta$ by
\begin{equation} \label{eq:H}
\H := \epsilon^2 L + |\d\gamma|^2.
\end{equation}
The operator $\H$ is the Schr\"odinger operator for a unit-mass
particle on $\Theta$ in a potential field $V = |\d\gamma|^2 +
\epsilon^2 \kappa$ with kinetic energy $T = \frac{1}{2}\langle
\mu(p),p \rangle$ induced by the sub-riemannian metric and Planck
constant $\hbar = 8 \epsilon^2$. From Theorem \ref{thm:optprior} it is
apparent that $R_{\epsilon}(\gb_{\epsilon}|\omega^2) = \epsilon^2 \int
\d\theta\, \omega \cdot \H\omega + O(\epsilon^6)$.

\begin{theorem} \label{thm:optp}
Let $\alpha_{\epsilon}$ be the largest eigenvalue of $\H$ with
eigenfunction $\omega = \omega_{\epsilon}$ which has $\int \d\theta\,
\omega_{\epsilon}^2 = 1$. Then
\begin{equation} \label{eq:bropt}
R_{\epsilon}(\gb_{\epsilon};\omega_{\epsilon}^2) = \epsilon^2
\alpha_{\epsilon} + O(\epsilon^6).
\end{equation} 
Let $\alpha$ be the largest eigenvalue of $L$ with eigenfunction
$\omega$ normalized so that $\int \d \theta\, \omega^2 = 1$. Then
\begin{equation} \label{eq:brkkopt}
\briskk{\gb_{\epsilon};\omega^2} = \epsilon^4
\alpha + O(\epsilon^6),
\end{equation}
and 
\begin{equation} \label{eq:ropt}
r(\Theta)=\alpha.
\end{equation}
\end{theorem}

\noindent
{\bf Remarks}. 1/ In Theorem~\ref{thm:optp}, it is assumed that $\H$
(resp. $L$) does possess a largest eigenvalue. Non-compactness of
$\Theta$ may negate this assumption; it may also be negated by
properties of the singular distribution $E$. A reformulation of the
theorem in the event that $\H$ (resp. $L$) has no largest eigenvalue
is clear. 2/ When are the eigenvalues of $\H$ (resp. $L$) constant?
This depends on the {\em accessibility} property of the singular
distribution $E$. If sections of $E$ generate $T\Theta$ under repeated
Lie brackets, then H\"ormander has shown that $\H$ (resp. $L$) is
hypoelliptic. At the opposite extreme, the distribution $E$ might be
integrable, in which case the eigenvalues of $\H$ (resp. $L$) will
vary from leaf to leaf. The optimal prior in this latter case is a
singular function (a distribution, in the functional-analytic sense)
concentrated on the leaf with the largest eigenvalue. 3/ The operator
$\H$ is a singular perturbation of a multiplication operator, so one
generally cannot na\"{\i}vely expand $\alpha_{\epsilon}$ in a power
series. However, when $|\d\gamma|=\alpha_0$ is constant, the
na\"{\i}ve idea is correct. In this case, one sees that
$\alpha_{\epsilon} = \alpha_0 + \epsilon^2 \alpha$ where $\alpha$
is the largest eigenvalue of $L$ (modulo the remarks in 1/). 4/
Important special cases include $\gamma$ being a riemannian submersion
or immersion.

\subsection{Optimal priors, II}
As noted in the beginning of this section, there is no natural reason
why one should choose $\d \theta$ as the flat prior. Let us
investigate the effect of choosing the flat prior to be $\dv = a^2 \d
\theta$. With $a\eta = \omega$, one computes from equations
(\ref{eq:sublap},\ref{eq:sublapa}) that
\begin{align} \label{eq:simpa}
4\int \d \theta\, \omega \Delta_{\sr{E}} \omega 
&= 4 \int \dv\,
\left\{ \eta^2 |\d \log a|^2 + 2 \eta \mu(\d \log a, \d \eta) + \eta
\Delta_{a^2\cdot \sr{E}} \eta \right\}, \nonumber\\ 
&= 4 \int \d \theta \times a^2 \times
\left\{ \eta^2 |\d \log a|^2 + \eta
\Delta_{\sr{E}} \eta \right\},
\end{align}
where $|\d \log a|^2 = \mu(\d \log a,\d \log a)$. Define
\begin{align}
\kappa_a & = \kappa + |\d \log a|^2, \label{eq:kappadefa}\\
L_a \eta & = \left(4\Delta_{\sr{E}}+\kappa_a \right) \eta. \label{eq:Ldefa}
\end{align}
From this discussion and the results of the previous section, the
following is clear.

\begin{theorem} \label{thm:optpriora}
The Bayesian risk of $\gb_{\epsilon}$ with prior $\lambda \d
\theta = \eta^2 \d \nu$ ($\d \nu = a^2 \d \theta$) equals
\begin{equation} \label{eq:optpriora}
R_{\epsilon}(\gb_{\epsilon};\lambda) = \epsilon^2 \int \dvol\, 
\eta^2\, |\d\gamma|^2 + \epsilon^4 \int \dvol\, \eta \cdot
L_a \eta + O(\epsilon^6),
\end{equation}
while
\begin{equation} \label{eq:optpriora2}
\briskk{\gb_{\epsilon};\lambda} = \epsilon^4 \int \dvol\, \eta \cdot
L_a \eta + O(\epsilon^6).
\end{equation}
\end{theorem}

\noindent
Since the operator $L_a$ is self-adjoint with respect to the inner
product determined by $\d \theta$, one knows that the prior that
maximizes $\briskk{\gb_{\epsilon};\lambda}$ occurs at a solution to
the eigenvalue problem
\[
L_a \eta = \alpha a^2 \eta. \eqno (\$)
\]

\begin{theorem} \label{thm:optpa}
Let $\alpha$ be the largest eigenvalue of the eigenvalue problem (\$)
with eigenfunction $\eta$ normalized so that $\int \dvol\, \eta^2 =
1$. Then, with $\lambda = a^2 \eta^2$,
\begin{equation} \label{eq:brkkopta}
\briskk{\gb_{\epsilon};\lambda} = \epsilon^4
\alpha + O(\epsilon^6),
\end{equation}
and
\begin{equation} \label{eq:rstaropt}
r^*(\Theta)=\alpha.
\end{equation}
\end{theorem}

$r^*(\Theta)$ is defined in equation (\ref{eq:reps2}).

\section{Applications} \label{sec:apps}
There are several cases in which the formulas of Theorem
\ref{thm:optp} yield especially nice results.

\subsection{Riemannian immersions} 
Recall that $\phi : (M,g) \to (N,h)$ is a riemannian immersion if
$\phi^*h = g$. If $\gamma : (\Theta,{\bf g}) \to (\Lambda,{\bf h})$ is
a riemannian immersion, then the riemannian structure and the induced
sub-riemannian structure coincide, while $|\d\gamma|^2 = \dim \Theta$
is constant. Remark 3/ following Theorem \ref{thm:optp} shows that

\begin{corollary} \label{cor:rimmersion}
Let $L = 4\Delta + \kappa$, where $\Delta$ is the laplacian of
$(\Theta,{\bf g})$ and $\kappa$ is defined in equation
(\ref{eq:kappadef}). Let $\alpha$ be the largest eigenvalue of $L$
with eigenfunction $\omega$ of unit $L^2$-norm. Then
$\alpha_{\epsilon} = \dim \Theta + \epsilon^2 \alpha$,
$\omega_{\epsilon} = \omega$ and
\begin{align*}
R_{\epsilon}(\gb_{\epsilon};\omega_{\epsilon}^2) &= \epsilon^2 ( \dim
\Theta + \epsilon^2 \alpha) + O(\epsilon^6),\\
&= \briskk{\gb_{\epsilon};\omega_{\epsilon}^2}.
\end{align*}
\end{corollary}

There are two interesting special cases of this corollary: when
$\gamma = {\rm id}_{\Theta}$ and when $\gamma = \iota$ (the inclusion
map of $\Theta$ into $\E$). By corollary \ref{cor:rimmersion}, the
sub-laplacian is the same in each case. However, the curvatures of the
identity map differ substantially from those of the inclusion map. One
sees that for $x$ in neighbourhood of $\Theta$
\begin{equation} \label{eq:goi}
g_{\epsilon}(x) = \begin{cases} \exp_{\pi(x)}\left( 2\epsilon^2 \nabla
\log |\omega| \right) + O(\epsilon^4) & {\rm if\ } \gamma = {\rm
id}_{\Theta},\\ \pi(x) + \epsilon^2\left( \tau(\iota) +
2\nabla\log|\omega| \right) + O(\epsilon^4)& {\rm if\ } \gamma =
\iota.
\end{cases}
\end{equation}
The tension field of the inclusion map $\tau(\iota)$ is $\dim\Theta$
times the mean curvature vector field -- in particular, it is normal
to $\Theta$ -- so $g_{\epsilon}(x) \not \in \Theta$ in the second
case. It should be noted that $\omega$ is not the same function in
each line. The curvature term $\kappa$ equals
\begin{equation} \label{eq:kappai}
\kappa = \begin{cases} \frac{5}{3}\, |\nabla\d\iota|^2 - \frac{2}{3}\,
|\tau(\iota)|^2 & {\rm if\ } \gamma = {\rm id}_{\Theta},\\
\frac{3}{2}\, |\nabla\d\iota|^2 - |\tau(\iota)|^2 & {\rm if\ } \gamma
= \iota.
\end{cases}
\end{equation}
While the two estimation problems are incomparable, strictly speaking,
it is interesting to observe that $\kappa$ -- and consequently, the
dominant eigenvalue of $L$ and bayesian risk -- is least for the
estimator of the inclusion map. This comes with an expense: the
estimator of $\iota$ does not take values that are on $\Theta$, while
the estimator of the identity is forced to do so.

B. Levit, in his unpublished {\em Habilitation} thesis, computes $\kappa$
in the case where $\gamma = \iota$. His calculations are carried out
in a system of local coordinates, which masks the difference between
the inclusion and the identity map. The present paper's formalism,
based on the Maclaurin series, clarifies these differences and
explains why the earlier estimator takes values off the manifold
$\Theta$. The formulas in equation \ref{eq:kappai} are proven in the
next section.

\subsubsection{Calculations for riemannian immersions} 
Let $(M,g) \stackrel{\phi}{\longrightarrow} (N,h)$ be a riemannian
immersion.  For the present calculations, it may be assumed that $N$
is a riemannian vector bundle over $M$ and $\phi$ is the inclusion of
the zero section. The projection map $N\to M$ is denoted by $\pi$.
$N$ is naturally identified with the normal bundle of $M$. The tangent
bundle to $N$ along $M$ is denoted by $T_{M}N = TM \oplus N$.

\begin{lemma} \label{lem:nabladpi}
Let $p \in M$ and $x+y,u+v \in T_pM \oplus N_p$. Then
$$\nabla\d\pi(x+y,u+v) = {\rm B}'_xv + {\rm B}'_uy,$$ where ${\rm
B}_{\bullet} : T_pM \to N_p$ is defined by ${\rm B}_{\bullet} =
\nabla\d\phi(\bullet,\, \cdot\,)$ and ${\rm B}'_{\bullet} : N_p \to
T_pM$ is the transposed map.
\end{lemma}

\begin{proof} 
For a smooth vector field $x$ on $M$, let $\tilde{x}$ be a smooth
vector field on $N$ that equals $x$ at $M$.  The Levi-Civita
connection on $M$ (resp. $N$) is $\nabla$
(resp. $\tilde{\nabla}$). Since $\phi$ is a riemannian immersion,
$\nabla_xy = \d\pi( \tilde{\nabla}_{\tilde{x}}\tilde{y}$). From this
fact it follows that $\nabla\d\pi\, |\, T_pM$ vanishes. On the other
hand, since $\pi \circ \exp_p\, |\, N_p = p$, $\nabla\d\pi\, |\, N_p$
vanishes. Finally, if $x\in T_pM$ and $v\in N_p$, then
$\nabla\d\pi(x,v) = -\d\pi(\tilde{\nabla}_xv)$ since
$\d\pi(v)=0$. Therefore, if $z \in T_pM$ then
\begin{align*}
\langle z, \nabla\d\pi(x,v) \rangle &= - \langle z,
\d\pi(\tilde{\nabla}_xv) \rangle = -\langle z, \tilde{\nabla}_xv
\rangle = \langle \tilde{\nabla}_{\tilde{x}}\tilde{z}, v \rangle,\\ &=
\langle v, \tilde{\nabla}_{\tilde{x}}\tilde{z} - \nabla_xz \rangle =
\langle v, \nabla\d\phi(x,z) \rangle,\\ &= \langle {\rm B}'_x v, z
\rangle.
\end{align*}
This completes the proof, since $\nabla\d\pi$ is bilinear.
\end{proof}

Let us compute the riemannian curvature tensor of $M$ in terms of that
of $N$ and the curvature of the immersion $\phi$. From the fact that
the Levi-Civita connection on $M$ is obtained by orthogonally
projecting the connection of $N$, we have that for all vector fields
$x,y,z$ on $M$
\begin{align*}
\Ri^M_{x,y}z &= \nabla_x\left( \nabla_yz \right) - \nabla_y\left(
\nabla_x z \right) - \nabla_{[x,y]}z,\\
&= \d\pi\left( \nabla_{\tilde{x}}\d\pi \cdot
\tilde{\nabla}_{\tilde{y}}\tilde{z} + \tilde{\nabla}_{\tilde{x}}
\tilde{\nabla}_{\tilde{y}}\tilde{z} - \nabla_{\tilde{y}}\d\pi \cdot
\tilde{\nabla}_{\tilde{x}}\tilde{z} -
\tilde{\nabla}_{\tilde{y}}\tilde{\nabla}_{\tilde{x}}\tilde{z} -
\tilde{\nabla}_{[\tilde{x},\tilde{y}]}\tilde{z} \right),\\
&= \d\pi\left(\, \Ri^N_{\tilde{x},\tilde{y}}\tilde{z}\,\right) +
\left({\rm B}'_x{\rm B}_y - {\rm B}'_y {\rm B}_x \right)z,
\end{align*}
where we have used the identity $\nabla_{\tilde{x}} \d\pi \cdot \tilde{\nabla}_{\tilde{y}} \tilde{z} = 
{\rm B}'_x \circ (1-\d\pi)\cdot \tilde{\nabla}_{\tilde{y}}\tilde{z}$, since $\nabla\d\pi$ vanishes on the 
horizontal part. Since $(1-\d\pi)\cdot \tilde{\nabla}_{\tilde{y}}\tilde{z} = \nabla\d\phi(y,z)$, this
demonstrates the final line.

From this equation, it follows that
\begin{equation} \label{eq:ricphi}
\Ric{}^M(x) = \d\pi\left(\, \sum_i \Ri^N_{\tilde{x},\tilde{e_i}} \tilde{e_i}\, \right) + \sum_i \left( {\rm B}'_x {\rm B}_{e_i}
- {\rm B}'_{e_i} {\rm B}_x \right) e_i,
\end{equation}
where $x \in T_pM$ and $e_i$ is an orthonormal basis of $T_pM$. Application of equation \ref{eq:ricphi} to equation \ref{eq:ricdphi} yields
\begin{equation} \label{eq:dphiricdphi}
\langle \d\phi, \Ric{\d\phi} \rangle = |\nabla\d\phi|^2 - |\tau(\phi)|^2.
\end{equation}

The scalar curvature of $M$ is the trace of the Ricci tensor, which
equals $\sum_{i,j} \langle e_i, \Ri^N_{e_i,e_j}e_j \rangle +
|\tau(\phi)|^2 - |\nabla\d\phi|^2$, where we omit the $\tilde{\
}$. When $\phi$ is the inclusion $\iota$ of $M$ into $\E$, we see
that
\begin{equation} \label{eq:scal}
\scal_M = |\tau(\iota)|^2 - |\nabla\d\iota|^2.
\end{equation}
To compute $\langle \d\phi, \nabla\tau(\phi) \rangle$, note that since
$\tau(\phi)$ is orthogonal to $T_pM$, $\langle \d\phi\cdot e_i ,
\nabla_{e_i} \tau(\phi) \rangle + \langle \nabla_{e_i} \d\phi \cdot
e_i, \tau(\phi) \rangle$ vanishes for all $i$. Therefore
\begin{equation} \label{eq:tau}
\langle \d\phi, \nabla\tau(\phi) \rangle = -|\tau(\phi)|^2.  
\end{equation}

\bigskip
\begin{proposition} \label{prop:rim} {[Following the notation of section 4.]} 
Let $\gamma$ be a riemannian immersion. Then $\Delta_{\sr{E}}$ is the
laplacian of $(\Theta,{\bf g})$, and
\begin{equation} \label{eq:kappaimmer}
\kappa = \frac{5}{3}\, | \nabla\d\iota |^2 -
\frac{2}{3}|\tau(\iota)|^2 - \frac{1}{6}\, | \nabla\d\gamma |^2 -
\frac{1}{3} |\tau(\gamma)|^2.
\end{equation}
If $\gamma$ is totally geodesic (iff $\nabla\d\gamma = 0$ iff
$\gamma(\Theta)$ is totally geodesic), then
\begin{equation} \label{eq:kappaimmertg}
\kappa = \frac{5}{3}\, | \nabla\d\iota |^2 -
\frac{2}{3}|\tau(\iota)|^2.
\end{equation}
\end{proposition}

\begin{proof} 
From equation \ref{eq:tau} and Theorem \ref{thm:optprior}, $\kappa =
\frac{1}{2}|\d\Gamma|^2 - \frac{2}{3} \langle \d\Gamma, \Ric{\d\Gamma}
\rangle - |\tau(\gamma)|^2$. It remains to compute the first two
terms.

\begin{enumerate}
\item Since $\Gamma = \gamma \circ \pi$, one sees that $\nabla\d\Gamma
= \nabla\d\gamma(\d\pi,\d\pi) + \d\gamma \cdot \nabla\d\pi$, which is
an orthogonal decomposition. Since $\gamma$ is a riemannian immersion,
$$|\nabla\d\Gamma|^2 = |\nabla\d\gamma|^2 + 2|\nabla\d\iota|^2.$$

\item From equation \ref{eq:ricdphi}, and the fact that
$\Ric{}^{\E}$ vanishes, $\langle \d\Gamma, \Ric{\d\Gamma} \rangle$
equals $\sum_{i,j} \langle u_i, \Ri^{\Lambda}_{u_i,u_j} u_j \rangle$
where $u_i = \d\gamma(e_i)$. Equation \ref{eq:ricdphi} and the
hypothesis that $\gamma$ is a riemannian immersion implies that this
equals $\langle \d\gamma, \Ric{\d\gamma} \rangle +
\scal_{\Theta}$. Equations \ref{eq:dphiricdphi} and \ref{eq:scal}
imply that
$$\langle \d\Gamma,\Ric{\d\Gamma} \rangle = |\nabla\d\gamma|^2 - |\tau(\gamma)|^2 - |\nabla\d\iota|^2 + |\tau(\iota)|^2.$$
\end{enumerate}
The two equations prove the formula for $\kappa$.
\end{proof}

\begin{corollary}
The formulas in equation \ref{eq:kappai} are correct.
\end{corollary}

\begin{proof} 
1/ If $\gamma = {\rm id}_{\Theta}$, then $\nabla\d\gamma = 0$ so
$\tau(\gamma)=0$, also. 2/ When $\gamma = \iota$ is the inclusion map,
the equation is clearly correct.
\end{proof}

\subsection{Riemannian submersions}
Recall that $M,g \stackrel{\phi}{\longrightarrow} N,h$ is a riemannian
submersion if $\d\phi\, |\, D\, :\, (D,g|D) \to (TN, h)$ is an
isometry where $D = (\ker \d\phi)^{\perp}$. In this case $|\d\phi|^2 =
\dim N$. Equation \ref{eq:ricdphi} implies that
\begin{equation} \label{eq:ricsub}
\langle \d\phi, \Ric{\d\phi} \rangle = -\scal_D + \scal_N \circ \phi,
\end{equation}
where $\scal_D$ is defined to be the trace of $\Ric{}^M | D$. Since
$|\d\phi|$ is constant, the identity $\frac{1}{2} \Delta |\d\phi|^2 =
\langle \d\phi, \Ric{\d\phi} \rangle - \nabla\tau(\phi) \rangle -
|\nabla\d\phi|^2$ implies that
\begin{equation} \label{eq:sffsub}
|\nabla\d\phi|^2 = \scal_N \circ \phi - \scal_D - \langle \d\phi, \nabla
\tau(\phi) \rangle.
\end{equation}

\bigskip
\begin{proposition} \label{prop:rsub}
{[Following the notation of section 4.]} Let $\gamma$ be a riemannian submersion. Then 
\begin{equation} \label{eq:kappasub}
\kappa = -\frac{1}{6}\, \scal_{\Lambda} \circ \gamma + |\tau(\gamma)|^2
+ \frac{3}{2} \langle \d\gamma, \nabla\tau(\gamma) \rangle.
\end{equation}
If $\gamma$ is harmonic (iff $\tau(\gamma)=0$) then
\begin{equation} \label{eq:kappasubh}
\kappa = -\frac{1}{6}\, \scal_{\Lambda} \circ \gamma.
\end{equation}
Let $e_i$ be an orthonormal frame of $E_{\theta}$. The sublaplacian 
$\Delta_{\sr{E}}$, when applied to a smooth function $f$, equals
\begin{equation} \label{eq:deltaEsub}
\Delta_{\sr{E}}f = \sum_{i=1}^{\dim E} \nabla^2_{e_i,e_i}
f \qquad {\rm at}\ \theta.
\end{equation}
\end{proposition}

\begin{proof} Equation
  \ref{eq:ricsub}, when applied to the submersion $\Gamma =
  \gamma\circ\pi$, yields $\langle \d\Gamma, \Ric{\d\Gamma} \rangle =
  \scal_{\Lambda} \circ \Gamma$ since $\E$ is flat. In addition,
  since $|\d\Gamma|$ is constant, equation \ref{eq:sffsub} and lemma
  \ref{lem:pi} implies that $|\nabla\d\Gamma|^2 = \scal_{\Lambda}
  \circ \Gamma -
  \langle \d\gamma, \nabla \tau(\gamma) \rangle.$ Equation
  \ref{eq:kappadef}, along with $\pi \circ \iota = {\rm id}_{\Theta}$,
  implies equation \ref{eq:kappasub}.

For the proof of equation \ref{eq:deltaEsub}, let $x^j$ be a system of
normal coordinates centred at $\theta$ and let $f_j = \frac{\partial\
\ }{\partial x^j} + O(|x|^2)$ be an orthonormal frame. Assume that
$e_i=f_i$ for $i=1,\ldots, \dim E$. The bundle map $\mu : T^* \Theta
\to T\Theta$ that characterizes the subriemannian structure is
written at $\theta$ as
$$\mu = \sum_{i}^{\dim E} f_i \otimes f_i = \sum_i^{\dim E}
\frac{\partial\ \ }{\partial x^i} \otimes \frac{\partial\ \ }{\partial
x^i} + O(|x|^2).$$ Since the riemannian metric ${\bf g} = \sum_{i=1}^n
\d x^i \otimes \d x^i + O(|x|^2)$, where $n=\dim \Theta$, the
riemannian volume form $\d\theta = (1 + O(|x|^2))\ \d x^1 \wedge
\cdots \wedge \d x^n$, and so the sub-laplacian at $\theta$ is
$$\Delta_{\sr{E}} = \sum_{i=1}^{\dim E} \frac{ \partial^2\ \ }{(\partial
  x^i)^2},$$ which equals $\sum_{i=1}^{\dim E} \nabla^2_{e_i,e_i}$ in
  invariant notation.
\end{proof}

\bigskip
\noindent{\bf Remark}. If one compares the formulas in
Propositions~\ref{prop:rim} and \ref{prop:rsub}, one sees that both
$\Delta_{\sr{E}}$ and $\kappa$ depend on the inclusion map $\iota$
when $\gamma$ is a riemannian immersion; when $\gamma$ is a riemannian
submersion $\Delta_{\sr{E}}$ and $\kappa$ do not depend on the
inclusion map $\iota$. In the latter case, the geometric information
carried by the riemannian submersion $\gamma$ subsumes that carried by
$\iota$.

\end{document}